\documentclass[12pt]{amsart}
\usepackage{amsmath}
\usepackage{amssymb}
\usepackage{ifthen}
\usepackage{datetime}
\usepackage[english]{babel}
\usepackage{epsfig}
\nonstopmode

\usepackage{mathrsfs}
\usepackage{graphicx}
\usepackage{latexsym}
\usepackage{verbatim}

\numberwithin{equation}{section}

\newcommand {\mo} {\,{\rm mod}\,}
\newcommand{\diam}{{\rm diam}\,}
\newcommand{\dist}{{\rm dist}\,}
\newcommand{\Int}{{\rm Int}\,}

\newcommand {\Sn} {{\overline{\mathbb R}^n}}
\newcommand {\Hn} {{\mathbb H^n}}
\newcommand {\ring} {{\mathcal R}}
\newcommand {\T} {{\mathcal T}}
\newcommand {\es} {{\mathcal S}}

\newcommand {\R} {\mathbb R}
\newcommand {\B} {\mathbb B}

\newcommand{\M}{\mathsf{M}}

\newcommand {\A} {\mathcal A}
\newcommand {\area} {A}
\newcommand {\aand} {\quad\text{and}\quad}

\newcommand {\Sp} {{\mathbb{S}^{n-1}_+}}

\theoremstyle{theorem}

\numberwithin{equation}{section}

\newenvironment{thm}[1][]{%
\refstepcounter{equation}%
\medskip%
\noindent%
\arabic{section}.\arabic{equation} {\bf Theorem}%
\ifthenelse{\equal{#1}{}}{}{ (#1)}%
{\bf .}
\itshape}{\medskip}

\newenvironment{lem}[1][]{%
\refstepcounter{equation}%
\medskip%
\noindent%
\arabic{section}.\arabic{equation} {\bf Lemma}%
\ifthenelse{\equal{#1}{}}{}{ (#1)}%
{\bf .}
\itshape}{\medskip}

\newenvironment{prop}[1][]{%
\refstepcounter{equation}%
\medskip%
\noindent%
\arabic{section}.\arabic{equation} {\bf Proposition}%
\ifthenelse{\equal{#1}{}}{}{ (#1)}%
{\bf .}
\itshape}{\medskip}

\newenvironment{rem}[1][]{%
\refstepcounter{equation}%
\medskip%
\noindent%
\arabic{section}.\arabic{equation} {\bf Remark}%
\ifthenelse{\equal{#1}{}}{}{ (#1)}%
{\bf .}
}{\medskip}

\renewcommand{\subsection}[1]%
{\smallskip \noindent \refstepcounter{equation}\arabic{section}.\arabic{equation} {\bf #1.}}

\begin{document}
\title[Modulus estimates of semirings]{Modulus estimates of semirings with applications \\ to boundary extension problems}

\author[A. Golberg]{Anatoly Golberg}
\address{School of Mathematical Sciences \\
Holon Institute of Technology \\
52 Golomb St., P.O.B. 305, Holon 5810201, Israel \\
ORCID 0009-0002-8785-7463
}
\email{golberga@hit.ac.il}

\author[T. Sugawa]{Toshiyuki Sugawa}
\address{Graduate School of Information Sciences \\
Tohoku University \\
Aoba-ku, Sendai 980-8579, Japan \\
ORCID 0000-0002-3429-5498
}

\email{sugawa@math.is.tohoku.ac.jp}

\author[M. Vuorinen]{Matti Vuorinen}
\address{Department of Mathematics and Statistics, University of Turku, FI-20014 Turku, Finland \\
ORCID 0000-0002-1734-8228
}
\email{vuorinen@utu.fi}

\begin{abstract}
In our previous paper \cite{GSV20}, we proved that the complementary components of a ring domain in $\mathbb{R}^n$ with large enough modulus may be separated by an annular ring domain and applied this result to  boundary correspondence problems under quasiconformal mappings. In the present paper, we continue this work and investigate boundary extension problems for a larger class of mappings.
\end{abstract}
\maketitle


{\small{\em{\centerline{Dedicated to the memory of Professor Lawrence Zalcman}}}}

\medskip

\bigskip

\section{Introduction}

Extremal problems of geometric function theory often lead to situations where the extremal configurations exhibit symmetry. Two classical examples of such extremal configurations are the ring domains of Gr\"otzsch and Teichm\"uller which provide lower bounds for the conformal capacities of the respective two classes of ring domains and have found many important applications in the theory of quasiconformal and quasiregular mappings in $\mathbb{R}^n, n\ge 2$ (\cite{GMP17}, \cite{HKV20}). Systematic
study of the capacities of these ring domains is carried out in \cite{AVV97}. In the planar case, the Teichm\"uller ring serves as an extremal case for the so-called Teichm\"uller theorem on the existence of an annular ring which separates the boundary components of a general ring domain. It seems, however, that the higher dimensional analogues of Teichm\"uller's theorem are less known.

In our previous paper \cite{GSV20}, we have extended Teichm\"uller's theorem and its semiring counterpart
to higher dimensions and, as examples of applications, given a
conformally invariant characterization of uniformly perfect sets in $\Sn.$
The following theorem \cite[Theorem~3.2]{GSV20} extends a variant of Teichm\"uller's theorem due to Avkhadiev and Wirths \cite{AW09}
to the $n$-dimensional case.

\begin{thm}\label{thm:teich}
Let $n\ge 2.$
Every ring domain $\mathcal R$ separating a given point $x_0$ in $\mathbb R^n$ and $\infty$
with $\mo \mathcal R>A_n$ contains an annular ring $\mathcal A$ centered at $x_0$
with $\mo \A\ge \mo \ring - A_n.$
Here $A_n$ is the constant defined in \eqref{eq:An} below and
this constant $A_n$ is sharp.
\end{thm}

Some other necessary results can be found in Section~\ref{Auxres}.
For their proofs we refer to \cite{GSV20}.

\medskip
In this paper, we emphasize that our approach allows us to weaken the regularity or quasiconformality assumptions
of the mappings. By definition every quasiconformal/quasiregular mapping $f:G\to\mathbb R^n$ of a domain $G\subset\mathbb R^n$ belongs to the Sobolev class $W^{1,n}_{\rm loc}(G).$ Moreover, $f$ is differentiable almost everywhere (a.e.) and possesses the Lusin $(N)$-property (preservation of zero measure sets) in $G.$ Arbitrary homeomorphisms of Sobolev class $W^{1,p}_{\rm loc}$ are differentiable a.e. only for $p>n-1,$ and the Lusin $(N)$-property holds for $p\ge n.$ For some details, see Section~\ref{Quasicon}.

Here we consider homeomorphisms of finite directional dilatations of the borderline class $W^{1,n-1}_{\rm loc}(G).$
Our technique involves various modulus bounds for semirings which rely on Theorem~\ref{thm:teich}. These results are presented in Section~\ref{Direc}.
In Section~\ref{Boundext} we establish various sufficient conditions on boundary extensions of the mappings considered, obtaining results which guarantee Lipschitz or weak H\"older type continuity of the extended mapping.


\section{Gr\"otzsch and Teichm\"uller rings and related estimates}\label{GroTeirings}

In this section we present some necessary notions connected to the conformal modulus of a curve/surface family and to the moduli of two distinguished rings named after Gr\"otzsch and Teichm\"uller in $\Sn,$ $n \ge 2.$ Here $\Sn$ denotes the extended Euclidean $n$-space $\R^n\cup\{\infty\},$
which is homeomorphic to the $n$-sphere $\mathbb{S}^n.$

\subsection{Modulus of curve/surface family}
Following \cite[9.2]{MRSY} (cf. \cite{Fug57}), we recall the notion of the modulus of a $k$-dimensional surface family (a curve family for $k=1$). Let $\omega$ be an open set in $\overline{\mathbb R}^k,$ $k=1,\ldots,n-1.$
A continuous mapping $S:\omega\to \mathbb R^n$ is called a $k$-dimensional surface $S$ in $\mathbb R^n.$
When $k=n-1,$ it is also called a hypersurface.
The number of preimages of a point $y,$ i.e. $N(S,y)={\rm card}\{x\in\omega\,:\, S(x)=y\}$ is said to be the \textit{multiplicity function} of $S$ at $y\in\mathbb R^n.$

By a $k$-dimensional Hausdorff area in $\mathbb R^n$ associated with a surface $S:\omega\to\mathbb R^n,$ we mean
\begin{equation*}
\area_S(B)\,=\,\area_S^k(B):=\,\int\limits_B N(S,y)\, d\mathcal H^k(y)
\end{equation*}
for every Borel set $B\subseteq\mathbb R^n.$
Here, $\mathcal H^k$ denotes the $k$-dimensional Hausdorff measure in $\mathbb R^n$
so normalized that $\mathcal H^k(I^k)=1,$ where $I^k=[0,1]^k\times\{0\}^{n-k}$ is the $k$-dimensional unit cube
embedded in $\R^n.$
The surface $S$ is called \textit{rectifiable}
if $\area_S(\mathbb R^n)<\infty.$

For a Borel function $g:\mathbb R^n\to[0,\infty],$ its integral over $S$ is defined by
\begin{equation*}
\int\limits_S g\,d\area^k\,:=\,\int\limits_{\mathbb R^n} g(y)N(S,y)\, d\mathcal H^k(y).
\end{equation*}

Let $\Gamma=\Gamma^k$ be a family of $k$-dimensional surfaces $S.$
A Borel measurable function $\varrho:\R^n\to [0,\infty]$ is called \textit{admissible} for $\Gamma^k$
if
\begin{equation*}
\int\limits_S \varrho^k\,d\area^k\, \ge\, 1
\end{equation*}
for every $S\in\Gamma^k.$
The (conformal) modulus of $\Gamma$ is defined to be
$$
\M(\Gamma)\,=\,\inf_\varrho \int_{\R^n}\varrho(x)^n\,dm_n(x)\,,
$$
where the infimum is taken over all admissible functions $\varrho$ on $\R^n$ for $\Gamma,$
and $m_n$ is the Lebesgue measure on $\R^n.$
If a property holds for all $S\in\Gamma\setminus \Gamma_0$ for some subfamily $\Gamma_0$
of $\Gamma$ with $\M(\Gamma_0)=0,$ we will say that the property holds for almost every $S\in\Gamma.$

\subsection{Rings}
Throughout our paper, a continuum will mean a connected, compact and non-empty set.
A continuum is said to be non-degenerate if it contains more than one point.
A continuum $C\subsetneq\Sn$ is called {\it filled} if $\Sn\setminus C$ is connected.
For a pair of disjoint filled continua $C_0$ and $C_1$ in $\Sn,$
the set $\ring=\Sn\setminus(C_0\cup C_1)$ is open and connected and
will be called a ring domain or, simply, a {\it ring} and sometimes denoted by $\ring(C_0, C_1).$
The ring $\ring$ is said to have non-degenerate boundary if each component $C_j$
is a non-degenerate continuum.
We will say that $\ring(C_0,C_1)$ separates a set $E$ if $\ring\cap E=\varnothing$ and if
$C_j\cap E\ne\varnothing$ for $j=0,1.$
In the sequel, when $\ring\subset\R^n,$
we will assume conventionally that $\infty\in C_1$ unless otherwise stated.

Let $\Gamma_\mathcal R$ be the family of all curves joining $C_0$ and $C_1$ in $\mathcal R.$
Also, let $\Sigma_\mathcal R$ be the family of hypersurfaces $S$ in $\mathcal R$ separating
the boundary of $\mathcal R.$
These are dual to each other in the sense that the relation
$\M(\Gamma_\mathcal R)=\M(\Sigma_\mathcal R)^{1-n}$ holds
(see \cite{Gol10}).
Then the modulus (called also the module) of  $\mathcal R$ is defined by
\begin{equation*}
\mo \mathcal R=\left[\frac{\omega_{n-1}}{\M(\Gamma_\mathcal R)}\right]^{1/(n-1)}
=\omega_n^{1/(n-1)}\M(\Sigma_\mathcal R) \,,
\end{equation*}
where $\omega_{n-1}$ denotes the area of the $(n-1)$-dimensional unit sphere \cite[p.~IX]{Vai71}.

For the annular (spherical) ring $\A(a; r_0,r_1)=\{x\in\mathbb R^n\,:\,r_0<|x-a|<r_1\},$
we have $\mo \A(a;r_0,r_1)=\log(r_1/r_0)$ (see, e.g. \cite[pp. 22-23]{Vai71}).

A ring $\ring'$ is said to be a {\it subring} of a ring $\ring$ if $\ring'\subset\ring$ and if
$\ring'$ separates $\Sn\setminus\ring.$
By the monotonicity of the moduli of curve families, we have the inequality $\mo \ring'\le \mo\ring$
in this case.

\subsection{Gr\"otzsch and Teichm\"uller rings}
Two canonical rings are of special interest because of
the extremal features of their moduli.
The first one is called the Gr\"otzsch ring $R_{G,n}(s),$ $s>1,$ and defined by
\begin{equation*}
R_{G,n}(s)=\mathcal R(\overline{\mathbb B}^n, [se_1,\infty]).
\end{equation*}
Here and hereafter $\mathbb B^n$ denotes the unit ball centered at the origin,
$\overline{\mathbb B}^n$ is its closure, $e_1$ is the unit vector $(1,0,\ldots,0)$
in $\R^n$ and $[se_1,\infty]=\{te_1: s\le t<\infty\}\cup\{\infty\}.$
The second one is the Teichm\"uller ring $R_{T,n}(t),$ $t>0,$ and defined by
\begin{equation*}
R_{T,n}(t)=\mathcal R([-e_1,0], [te_1,\infty]).
\end{equation*}
The functions $\gamma_n(s)=\M(\Gamma_{R_{G,n}(s)})$ and
$\tau_n(t)=\M(\Gamma_{R_{T,n}(t)})$ are systematically studied in \cite{AVV97}.

Here we briefly recall the main properties of the moduli of the Gr\"otzsch and Teichm\"uller rings, see, e.g. \cite{AVV97}, \cite[5.4.1, pp.~181-182]{GMP17}, \cite[pp.~157-159]{HKV20}.

\begin{itemize}
\item Both $\gamma_n$ and $\tau_n$ are strictly decreasing and continuous functions.
\item Let $\ring$ be the ring $\ring(\overline\B^n,C_1)$
for a filled continuum $C_1$ with $y, \infty\in C_1$ in the domain $|x|>1.$
Then $\mo\ring\le\mo R_{G,n}(|y|).$
\item For filled continua $C_0, C_1$ with $0, -e_1\in C_0$ and $x_1,\infty\in C_1,$ $\mo \mathcal R(C_0,C_1)\le \mo R_{T,n}(|x_1|).$
\item The following functional identity holds for $t>0,$
\begin{equation*}
\mo R_{T,n}(t)\,=\,2 \mo R_{G,n}(s),\qquad s=\sqrt{t+1}\,.
\end{equation*}
\end{itemize}

To define two important constants, we make use of two real-valued functions $\Phi_n$ and $\Psi_n$ defined by
\begin{align*}
\log\Phi_n(s)&=\mo R_{G,n}(s)=\left[\frac{\omega_{n-1}}{\gamma_n(s)}\right]^{1/(n-1)}, \\
\log\Psi_n(t)&=\mo R_{T,n}(t)=\left[\frac{\omega_{n-1}}{\tau_n(t)}\right]^{1/(n-1)}\,.
\end{align*}

The Gr\"otzsch (ring) constant $\lambda_n,$ defined by
\begin{equation*}
\lambda_n\,:=\, \lim\limits_{s\to \infty} \Phi_n(s)/s\,,
\end{equation*}
admits the following bounds
\begin{equation*}
4\,\le\,\lambda_n \,\le\, 2^{n/(n-1)}e^{n(n-2)/(n-1)}
\end{equation*}
and has numerous applications to various fields of Real and Complex Analysis.
Note that $\lambda_2=4$ and the exact value of $\lambda_n$ is unknown for $n\ge 3;$ see \cite{AVV97, GMP17}.

\medskip
The quantity $A_n$ mentioned in Theorem~\ref{thm:teich} is defined by
\begin{equation}\label{eq:An}
A_n\,=\,\sup_{1<t<+\infty}\big[\mo R_{T,n}(t)-\log t\big]
\,=\,\sup_{1<t<+\infty}\log\frac{\Psi_n(t)}{t}\,.
\end{equation}
Moreover, the number $A_n$ admits the estimate (see \cite[Theorem 3.2]{GSV20}):
\begin{equation*}
A_n\le
2\log\frac{(1+\sqrt 2)\lambda_n}2
=\log \frac{(3+2\sqrt{2})\lambda_n^2}{4}.
\end{equation*}
When $n=2,$ it is known that $A_2=\pi.$


\section{Auxiliary results}\label{Auxres}

In Introduction we have formulated the multidimensional counterpart of the Teichm\"uller theorem. This theorem is also crucial for the present paper and its proof  can be found in our previous paper \cite{GSV20}. We also apply  the following results of the same manuscript and provide them here for convenience of the reader.

\subsection{Semirings}
Following our previous paper \cite{GSV20}, the standard model for ``semiring" is the upper half of the {\it closed} ring
$$
\T_R\,=\,\{x\in\Hn: 1\le |x|\le R\}
$$
for $1<R<+\infty.$
Here and hereafter $\Hn$ denotes the upper half space $\{x=(x_1,\dots,x_n): x_n>0\}.$
The semiring $\T_R$ has two distinguished boundary components
\begin{equation*}
\partial_0\T_R=\{x\in\Hn: |x|=1\}\qquad \text{and}\qquad
\partial_1\T_R=\{x\in\Hn: |x|=R\}
\end{equation*}
relative to  $\Hn\,,$ which are homeomorphic to
the $(n-1)$-dimensional open ball $\B^{n-1}.$
Let $\Gamma(R)$ denote the family of arcs (curves) $\gamma:[0,1]\to \T_R$
joining $\partial_0\T_R$ and $\partial_1\T_R$ in $\T_R.$
Thanks to \cite[7.7]{Vai71}, we obtain the formula
\begin{equation}\label{eq:modulus}
\M(\Gamma(R))=\frac{\omega_{n-1}}{2}\left(\log R\right)^{1-n}.
\end{equation}
Let $\Sigma(R)$ denote the family of surfaces $S:\B^{n-1}\to\T_R$ which are
proper maps and the images separate $\partial_0T_R$ from $\partial_1T_R$ in $\T_R.$
By the symmetry principle, we also have
$\M(\Sigma(R))=(\omega_{n-1}/2)^{1/(1-n)}\log R.$

A subset $\es$ of $\Sn$ is called a {\it semiring} if it is homeomorphic to $\T_R$ for some $R>1.$
We denote by $\Gamma_\es$ the family of the image curves of $\Gamma(R)$
under a homeomorphism $f:\T_R\to \es.$
In other words, $\Gamma_\es$ consists of curves joining the distinguished boundaries
 $\partial_0\es=f(\partial_0\T_R)$ and $\partial_1\es=f(\partial_1\T_R)$ in $\es.$
Similarly, we denote by $\Sigma_\es$ the image surfaces of $\Sigma(R)$ under $f:\T_R\to\es.$
Note that $\M(\Gamma_\es)$ and $\M(\Sigma_\es)$ do not change under conformal transformations.
Moreover, as in the case of rings, by \cite[Thm~3.13]{Ziem67},
we have the relation
$$
\M(\Sigma_S)=\left[\M(\Gamma_S)\right]^{1/(1-n)}.
$$
We now define the modulus of the semiring $\es$ by
\begin{equation}\label{eq:Ziem}
\mo \es=\left[\frac{\omega_{n-1}}{2\M(\Gamma_\es)}\right]^{1/(n-1)}
=\left(\frac{\omega_{n-1}}{2}\right)^{1/(n-1)}\M(\Sigma_\es).
\end{equation}
In particular, $\mo\T_R=\log R$ by virtue of \eqref{eq:modulus}.

\subsection{Properly embedded semirings}
Let $G$ be a proper subdomain of $\Sn.$
A semiring $\es$ in $G$ is said to be {\it properly embedded} in $G$ if
$\es\cap C$ is compact whenever $C$ is a compact subset of $G.$
That is to say, $\es$ is a properly embedded semiring in $G$ if and only if
some (and hence every) homeomorphism $f:\T_R\to \es$ is proper as considered to be
a map $f:\T_R\to G.$
Note that $\partial_0\es$ and $\partial_1\es$
are properly embedded $(n-1)$-dimensional open balls in $G.$
(Though there is no canonical way to label $\partial_0\es$ and $\partial_1\es$
to the connected components of $\partial\es$ in $G,$ we take the labels given by
a proper embedding $f:\T_R\to G$ and fix them for convenience.)

From now on, we consider a semiring $\es$ properly embedded in $\B^n$ by a mapping
$f:\T_R\to\es\subset\B^n.$
Then $\B^n\setminus\es$ is an open subset of $\B^n$ consisting of two components $V_0$ and $V_1$
for which $V_0\cap\partial_1\es=\varnothing$ and $V_1\cap\partial_0\es=\varnothing.$
The following separation lemma \cite[Lemma~4.3]{GSV20} will be applied in the last section of our paper.

\begin{lem}\label{lem:sep}
Let $\es$ be a properly embedded semiring in $\B^n.$
Then $\mo \es>0$ if and only if the Euclidean distance $\delta=\dist(V_0, V_1)$ between
$V_0$ and $V_1$ is positive.
Moreover, in this case, the double $\hat\es:=\Int\es\cup U \cup \Int\es^*$ of $\es$
is a ring with
$\mo\hat\es=\mo\es,$ where $\es^*$ is the reflection of $\es$ in $\partial\B^n$
and $U=\partial\B^n\setminus(\overline V_0\cup\overline V_1).$
\end{lem}

\medskip
For $\xi\in\partial\B^n$ and $0<r_0<r_1<+\infty,$ we define a properly embedded semiring
$$
\T(\xi; r_0, r_1)=\left\{x\in\B^n: r_0\le\frac{|x-\xi|}{|x+\xi|}\le r_1\right\}
$$
in $\B^n$ bounded by two Apollonian spheres.
Then we have the formula $\mo\T(\xi; r_0, r_1)=\log(r_1/r_0)$ (see Lemma~4.2 in \cite{GSV20}).
The next proposition and theorem have been obtained in \cite{GSV20}; see Proposition~5.3 and Theorem~5.5, respectively.

\begin{prop}\label{prop:lim}
Let $f:\B^n\to\B^n$ be a homeomorphism and $\xi\in\partial\B^n.$
The mapping $f$ extends continuously to the point $\xi$ if
$$
\lim_{r\to0+}\mo f(\T(\xi; r,R))=+\infty
$$
for some $R>0.$
\end{prop}

\begin{thm}\label{thm:exten}
A homeomorphism $f:\B^n\to\B^n$ extends to a homeomorphism $f:\overline\B^n\to
\overline\B^n$ if and only if for each $\xi\in\partial\B^n,$ there is an $R=R_\xi>0$ such that
$$
\lim_{r\to0+}\mo f(\T(\xi; r,R))=+\infty.
$$
\end{thm}

\medskip
We also need the following separation theorem \cite[Theorem~4.8]{GSV20}.

\begin{thm}\label{thm:sep}
Let $\es$ be a properly embedded semiannulus in $\B^n.$
Then the connected components $V_0$ and $V_1$ of $\B^n\setminus\es$ satisfy the inequality
\begin{equation}\label{eq:sep}
\min\{\diam V_0,\diam V_1\}\le Q_n \exp\left(-\frac12\mo\es\right),
\end{equation}
where $Q_n=4\exp(A_n/2)$ and $A_n$ is given in \eqref{eq:An}.
\end{thm}


\section{Quasiconformal, quasiregular mappings and their regularity properties}\label{Quasicon}


Due to the famous Liouville theorem, there are no conformal mappings in higher dimensions $n\ge 3$ except M\"obius transformations; see, e.g. \cite{Resh94}. The classes of quasiconformal mappings and their non-homeomorphic counterparts - quasiregular mappings (or mappings with bounded distortion) are substantially larger than the class of conformal mappings in $\mathbb{R}^n, n\ge 3,$ and many of the main geometric and topological properties of analytic functions in the complex plane have their counterparts in this $n$-dimensional function theory.

\subsection{Linear mappings}\label{ss:linear}
Following \cite[14.1]{Vai71}, for a real $n\times n$ invertible matrix $A,$ we define
\begin{equation}\label{eq:maxmin}
\|A\|\,=\,\sup_{h\in\R^n,~ h\ne 0}\frac{|Ah|}{|h|}\,=\,
\max_{|h|=1}|Ah|
\aand
l(A)\,=\,\inf_{h\in\R^n,~ h\ne 0}\frac{|Ah|}{|h|}\,=\,\min_{|h|=1}|Ah|\,.
\end{equation}
Then the quantities
\begin{equation*}
H_I(A)\,=\,\frac{|{\rm det}\,A|}{l(A)^n}\,,\qquad H_O(A)\,=\,\frac{\|A\|^n}{|{\rm det}\,A|}\,,\qquad H(A)\,=\,\frac{\|A\|}{l(A)}
\end{equation*}
are called the \textit{inner}, \textit{outer} and \textit{linear dilatation coefficients} of $A,$ respectively.

By linear algebra, the following inequalities
\begin{equation}\label{eq:relat}
H(A)\,\le\,\min\{H_I(A), H_O(A)\}\,\le\,H(A)^{n/2}\,\le\,\max\{H_I(A), H_O(A)\}\,\le\,H(A)^{n-1}\,,
\end{equation}
hold; cf. \cite[14.3]{Vai71}.

\subsection{Quasiregular and quasiconformal mappings}
Following \cite[I.2]{Ric93}, recall the definitions of quasiregular and quasiconformal mappings. A mapping $f\,:\,G\to \mathbb R^n,$ $n\ge 2,$ of a domain $G$ in $\mathbb R^n$ is \textit{quasiregular} if
\begin{enumerate}
\item $f$ is in ${\rm ACL}^n(G),$ and
\item there exists a constant $K,$ $1\le K<\infty,$ such that
\end{enumerate}
\begin{equation}\label{eq:qr}
\|f'(x)\|^n\,\le\, K J_f(x)\qquad\text{a.e.~in~} G.
\end{equation}

The smallest $K$ in  (\ref{eq:qr}) is the \textit{outer dilatation} $K_O$ of $f$ in the domain $G.$
Here and hereafter $f'(x)$ denotes the Jacobian matrix of $f$ at $x,$ and $J_f(x)$ denotes its determinant.
Note that for continuous mappings the classes ${\rm ACL}^n(G)$ and $W^{1,n}(G)$ coincide; see, e.g. \cite[A.5]{HK14}.

If $f$ is quasiregular, then it is also true that
\begin{equation}\label{eq:qr1}
J_f(x)\,\le\, K' l(f'(x))^n\qquad\text{a.e.~ in~} G
\end{equation}
for some $K'\ge 1,$ where $l(f'(x))$ is defined in (\ref{eq:maxmin}).
The smallest $K'\ge 1$ in (\ref{eq:qr1}) is the \textit{inner dilatation} $K_I(f)$ of $f$ in $G$.
A quasiregular homeomorphism $f:G\to f(G)$ is called \textit{quasiconformal} \cite[I.2]{Ric93}.

In the case of a continuous, discrete and open mapping $ f : G \to \mathbb{R}^n,$ the linear dilatation is defined as follows.
If $x\in G,$ $0<r<\dist (x,\partial G),$ we set
\begin{equation*}
l(x,r)\,=\,l_f(x,r)\,=\,\inf\limits_{|y-x|=r}|f(y)-f(x)|\,,\quad
L(x,r)\,=\,L_f(x,r)\,=\,\sup\limits_{|y-x|=r}|f(y)-f(x)|\,.
\end{equation*}
The quantity
\begin{equation*}
H(x,f)\,=\,\limsup\limits_{r\to 0}\frac{L(x,r)}{l(x,r)}
\end{equation*}
is called the \textit{linear dilatation}.

We also say that a point $x$  is a \textit{regular} point of $f,$ if $f$ is differentiable at $x$ and $J_f(x)\ne 0.$
For a regular point $x\in G,$ $H(x,f)$ equals the linear dilatation $H(f'(x))$
of the mapping $f'(x)$ (see \S \ref{ss:linear}).

Below we list some of the main properties of  quasiregular/quasiconformal mappings relevant for us, see, e.g. \cite{HKV20, HK14, Resh89, Ric93}.

Let $f:G\to \mathbb R^n$ be a quasiregular mapping. Then
\begin{itemize}
\item \emph{$f$ is differentiable a.e.~in $G.$}
\item \emph{$f$ satisfies the Lusin $(N)$-property, i.e. $m(E)=0,$ $E\subset G,$ implies $m(f(E))=0.$}
\item \emph{$f$ is locally H\"older continuous with exponent $K^{1/(1-n)}$ where $\max\{K_I(f),K_O(f)\}\le K.$}
\item \emph{If $f$ is nonconstant, it is discrete, open and orientation-preserving.}
\item \emph{$f$ belongs to $W^{1,p}_{\rm loc}(G)$ with $p=p(n,K)>n.$ }
\end{itemize}

In the above, the orientation-preserving property is in the topological sense.
Also, it is known that $J_f>0$ a.e.~in $G$ for a non-constant quasiregular
map $f:G\to\R^n.$
This was first shown by Martio, Rickman and V\"ais\"al\"a \cite[Theorem 8.1]{MRV69}
(see also \cite[p. 48]{Ric93}).


\subsection{Regularity properties of $W^{1,p}$-homeomorphisms}
Following mainly \cite[Chapters 2 and 4]{HK14}, we recall some needed regularity properties of continuous/homeomorphic mappings of the Sobolev classes $W^{1,p}.$

\begin{itemize}
\item \emph{Every mapping $f\in W^{1,p}$ is differentiable a.e. when $p>n$ and $n\ge 2.$}
\item \emph{For $p=n$ there exist mappings $f\in W^{1,p}$ which are not continuous at any point, and, therefore, differentiable nowhere.}
\item \emph{Every homeomorphism of $W^{1,p},$ with $p>n-1$ for $n\ge 3$, $p\ge 1$ for $n=2$ is differentiable a.e.}
\item \emph{A continuous mapping $f\in W^{1,p}$ always satisfies the Lusin $(N)$-property with
respect to the $n$-dimensional Lebesgue measure when $p>n.$}
\item \emph{There exist continuous mappings $f\in W^{1,n}$ which fail to have the Lusin $(N)$-property.}
\item \emph{For homeomorphisms $f\in W^{1,n},$ the Lusin $(N)$-property holds.}
\item \emph{There are homeomorphisms of $W^{1,p},$ $p<n,$ which do not possess the Lusin $(N)$-property.}
\end{itemize}

For differentiability a.e., the borderline class $W^{1,n-1}$ is of special interest.
We need the following statement proved in \cite{Teng14}.

\begin{lem}\label{lem:Teng}
Let $G\subset\mathbb R^n$ be a domain for $n \ge 2.$
Suppose that $f\in W^{1,n-1}_{\rm loc}(G)$ is a
continuous, discrete and open mapping with (pointwise) inner dilatation $L_f(x)=H_I(f'(x))$
satisfying $L_f\in L^1_{\rm loc}(G).$
Then $f$ is differentiable a.e.~in $G.$
\end{lem}


\section{Directional dilatations and modulus estimates}\label{Direc}


\subsection{Directional dilatations}
We recall two directional characteristics in $\mathbb R^n$,
using the derivative of $f$ in a direction $h, h\ne 0$, at $x$,
given by
\begin{equation*}
\partial_h f(x)\,=\,\lim\limits_{t\to 0^+}\frac{f(x+th)-f(x)}{t}\,,
\end{equation*}
whenever the limit exist.
Note that $\partial_h f(x)=f'(x)h$ if $f$ is differentiable at $x.$

Let $f:\,G\to \mathbb R^n$ be an orientation-preserving homeomorphism and $x\in G$ be a regular point of $f.$
For a point $x_0\in\mathbb R^n$, we define the \textit{angular} and
\textit{normal dilatations} of the mapping $f$ at $x\in G$,
$x\ne x_0$ \textit{with respect to} $x_0$ by
\begin{equation*}
D_f(x,x_0)=\frac{J_f(x)} {\ell_f(x,x_0)^n},\quad\quad
T_f(x,x_0)=\biggl(\frac{\mathcal L_f(x,x_0)^n}{J_f(x)}\biggr)^{1/(n-1)},
\end{equation*}
respectively.
Here
\begin{equation*}
\ell_f(x,x_0)=\min\limits_{|h|=1}\frac{|\partial_h f(x)|}{|h\cdot u|}, \quad\quad
\mathcal L_f(x,x_0)=\max\limits_{|h|=1}\bigl(|\partial_h f(x)||h \cdot u|\bigr),
\end{equation*}
and $u=(x-x_0)/|x-x_0|$.
The dilatations $D_f(x,x_0)$ and $T_f(x,x_0)$ are both measurable in $x\in G$.

Following the notations in \cite{GG09} we denote by $L_f(x)$ and $K_f(x)$ the inner and outer dilatations of $f$ at a regular point $x$ of $f,$ respectively.
Namely, $L_f(x)=H_I(f'(x))$ and $K_f(x)=H_O(f'(x)).$
Then the chain of inequalities
\begin{equation*}
l(f'(x))\le \ell_f(x,x_0)\le |\partial_u f(x)|\le \mathcal L_f(x,x_0)\le \|f'(x)\|
\end{equation*}
implies
\begin{equation}\label{eq.2.2}
\frac1{K_f(x)}\le D_f(x,x_0)\le L_f(x).
\end{equation}
The normal dilatation $T_f(x,x_0)$ has tighter bounds than
$D_f(x,x_0)$, since
\begin{equation}\label{eq.2.3}
\frac1{K_f(x)}\le \frac1{L_f(x)^{1/(n-1)}}\le T_f(x,x_0) \le
K_f(x)^{1/(n-1)}\le L_f(x).
\end{equation}

The dilatations $D_f(x,x_0)$ and $T_f(x,x_0)$ for the
multidimensional case have been introduced in \cite{GG09} and
\cite{Gol10}, respectively. Note that  the angular and normal dilatations range both
between 0 and $\infty$, while the classical dilatations are always
greater than or equal to 1, cf. (\ref{eq:relat}).
Clearly, $L_f(x)=K_f(x)\equiv 1$ for conformal mappings $f$ and, therefore, both directional dilatations also are equal to 1. But not vice versa.
Observe also that these directional dilatations provide a reasonable kind of flexibility,
although their concrete evaluations are much more complicated than those of classical ones; see, e.g. \cite{Gol18}.
The latter fact can be illustrated by the rotation of the punctured ball $\mathbb B^n\setminus\{0\}$ given by
\begin{equation*}
f(x)\,=\,(x_1\cos\theta-x_2\sin\theta, x_2\cos\theta+x_1\sin\theta,
x_3,...,x_n)\,,
\end{equation*}
with $x=(x_1,...,x_n)$ and $\theta=\log(x_1^2+x_2^2)$.
This mapping preserves the volume: $J_f(x)\equiv 1.$
By a straightforward calculation, one obtains
$L_f(x)=K_f(x)=(1+\sqrt{2})^n$ for $x\ne 0,$ and this mapping is quasiconformal in $\mathbb B^n\setminus\{0\}.$
However, $D_f(x,0)=1$ at all points $x$ of $\mathbb B^n\setminus\{0\};$ see \cite{GG09}.


\subsection{Main Lemma}
We consider the semiring
$\es=\es(x_0;r,R)=\{x\in\mathbb H^n: r\le |x-x_0|\le R\}$ for $x_0 \in \partial \mathbb H^n.$
We recall that $\mo \es=\log(R/r).$
The modulus distortion under quasiconformal and quasiregular mappings plays an essential role in geometric function theory; see, e.g. \cite{Sug10} and \cite{ZYRH24}.
The following lemma gives upper and lower bounds for the distortions of moduli of semirings for homeomorphisms of Sobolev class $W^{1,n};$ cf. \cite[Corollary 5.1]{Gol11}.

\medskip
\begin{lem}\label{lem:modest}
Let $f$ be an orientation-preserving homeomorphism of $\Hn$ onto a domain in $\R^n.$
Suppose that $f$
belongs to the Sobolev class $W^{1,n-1}_{\rm loc}(\Hn)$ and posesses Lusin's $(N)$ and $(N^{-1})$-properties with respect to the $n$-dimensional Lebesgue measure $m=m_n.$
Suppose further that the inner dilatation $L_f(x)$ is locally integrable in the semiring
$\es=\es(x_0;r_0,r_1)$ for some $x_0\in\partial\mathbb H^n,$ and
for almost every hypersurface $S\in \Sigma_\es$ the restriction $f|_S$ satisfies
the $(N^{-1})$-property with respect to $(n-1)$-dimensional Hausdorff measure.
Then
\begin{equation}\label{eq1est}
\begin{split}
\left(\frac{2}{\omega_{n-1}\log(r_1/r_0)}\int\limits_\es
\frac{D_f(x,x_0)}{|x-x_0|^n}\,dm_n(x)\right)^{1/(1-n)}\, &
\le\,\frac{\mo f(\es)}{\mo \es}\,\\&\le\,
\frac{2}{\omega_{n-1}\log(r_1/r_0)}\int\limits_\es
\frac{T_f(x,x_0)}{|x-x_0|^n}\,dm_n(x)\,,
\end{split}
\end{equation}
and under the additional assumption $\mo \es \ge \mo f(\es),$
\begin{equation}\label{eq2est}
\begin{split}
-\frac {2}{\omega_{n-1}} \int\limits_{\es}\frac{T_f(x,x_0)-1}{|x-x_0|^n}\,dm_n(x)
\,&\le\,\mo \es-\mo f(\es)
\,\\&\le\, \frac {2}{\omega_{n-1}} \int\limits_{\es}\frac{D_f(x,x_0)-1}{|x-x_0|^n}\,dm_n(x)\,.
\end{split}
\end{equation}
\end{lem}

\begin{rem}
We note that $\mo\es=\log(r_1/r_0)$ by \eqref{eq:modulus}.
If we introduce the measure $d\nu_{x_0}(x)=|x-x_0|^{-n}dm_n(x),$
we have $\nu_{x_0}(\es)=(\omega_{n-1}/2)\log(r_1/r_0).$
Hence, the estimates \eqref{eq1est} can be written in the form
$$
\left(\frac1{\nu_{x_0}(\es)}\int\limits_\es
D_f(x,x_0)\,d\nu_{x_0}(x)\right)^{1/(1-n)}\,
\le\,\frac{\mo f(\es)}{\mo \es}\,\le\,
\frac{1}{\nu_{x_0}(\es)}\int\limits_\es
T_f(x,x_0)\,d\nu_{x_0}(x)\,.
$$
\end{rem}

\subsection{Proof of the first inequality in \eqref{eq1est}}
Denote by $\es_0$ the set of regular points of $f$ in $\es.$
By virtue of Lemma \ref{lem:Teng}, we find that the set $B_0=\es\setminus\es_0$ has
the $n$-dimensional Lebesgue measure zero, $m_n(B_0)=0.$
The Lusin $(N^{-1})$-property is equivalent to $J_f(x)\ne 0$ a.e.; cf. \cite{Pon95}. Now the $(N)$-property implies that also $m_n(f(B_0))=0.$ Note also that by (\ref{eq.2.2})--(\ref{eq.2.3}) the local integrability of $L_f(x)$ implies the same property for both directional dilatations $D_f(x,x_0)$ and $T_f(x,x_0).$

Recall that $\Gamma_\es$ is the family of curves which join the boundaries $|x-x_0|=r_0$ and $|x-x_0|=r_1$ in $\es.$
Let $\varrho\ge 0$ be a Borel function on $[r_0,r_1]$ such that
\begin{equation}\label{admis1}
\int\limits_{r_0}^{r_1} \varrho(t)\,dt\,=\,1.
\end{equation}
For any $y\in f(\es\setminus B_0)=f(\es)\setminus f(B_0)$ we define
\begin{equation*}
\varrho^*(y)\,=\,\varrho(|x-x_0|)\left(\frac{D_f(x,x_0)}{J_f(x)}\right)^{1/n}
=\frac{\varrho(|x-x_0|)}{\ell_f(x,x_0)}\,,
\end{equation*}
where $x=f^{-1}(y),$
and set $\varrho^*(y)=\infty$ for $y\in f(B_0),$ and $\varrho^*(y)=0$ otherwise.

We now claim that $\int_{\gamma^*} \varrho^*\,d\area^1\ge 1$
for almost every curve $\gamma^*\in f(\Gamma_\es)=\Gamma_{f(\es)}.$
Then it will follow that $\varrho^*$ is admissible for $\Gamma_{f(\es)}$
(see e.g. \cite[Theorem~9.1]{MRSY}).
Though this claim is found in \cite[Lemma~2.4]{GG09},
in order to see how the definition of $D_f(x,x_0)$ works,
we include some details of the proof here.
For almost every curve $\gamma^*=f(\gamma)\in \Gamma_{f(\es)}=f(\Gamma_\es),$
both $\gamma$ and $\gamma^*$ are rectifiable.
We now parametrize them as $x=\gamma(s)$ and $y=\gamma^*(s^*)$
by their length parameters so that
$|d\gamma(s)/ds|=1$ and $|d\gamma^*(s^*)/ds^*|=1$ a.e.
Noting that $h_1=d\gamma(s)/ds$ is a unit vector, we have
$$
\frac{d\gamma^*(s^*)}{ds}=f'(x)\frac{d\gamma(s)}{ds}=f'(x)h_1=\partial_{h_1}f(x)
$$
as long as $f$ is regular at $x=\gamma(s).$
Hence,
$$
\frac{ds^*}{ds}=\left|\frac{d\gamma^*}{ds}\right|=|\partial_{h_1}f(x)|\,.
$$
Since the quantity $r=|x-x_0|$ has the gradient
$$
\nabla r=\frac{x-x_0}{r}=:u,
$$
we have the expression
$$
\frac{dr}{ds}
=\nabla r\cdot\frac{d\gamma(s)}{ds}
=u\cdot h_1.
$$
Let $a$ and $a^*$ be the lengths of the curves $\gamma$ and $\gamma^*,$
respectively.
Then
\begin{align*}
\int\limits_{\gamma^*} \varrho^*\,d\area^1
&=\int_0^{a^*}\varrho^*(\gamma^*(s^*))ds^* \\
&=\int_0^a \frac{\varrho(|x-x_0|)}{\ell_f(x,x_0)}\frac{ds^*}{ds}ds \\
&=\int_{r_0}^{r_1} \frac{\varrho(r)}{\ell_f(x,x_0)}\frac{ds^*/ds}{dr/ds}dr \\
&=\int_{r_0}^{r_1} \frac{\varrho(r)}{\ell_f(x,x_0)}
\frac{|\partial_{h_1}f(x)|}{h_1\cdot u}dr \\
&\ge\int_{r_0}^{r_1}\varrho(r)dr=1
\end{align*}
as required.

Using the chain rule and the Lusin $(N)$-property of $f,$ we have
\begin{equation*}
\M(\Gamma_{f(\es)})\,\le\,\int\limits_{f(\es)}\varrho^*(y)^n dm_n(y)\,
=\,\int\limits_\es \varrho(|x-x_0|)^n D_f(x,x_0)\,dm_n(x)\,.
\end{equation*}

Since $\varrho(t)=1/(t\log(r_1/r_0))$ satisfies (\ref{admis1}), the following estimate
is obtained:
\begin{equation}\label{Gamma}
\M(\Gamma_{f(\es)})\,\le\,\frac{1}{\log^n(r_1/r_0)}
\int\limits_\es\frac{D_f(x,x_0)}{|x-x_0|^n}\,dm_n(x)\,.
\end{equation}
Finally, in view of the first definition in \eqref{eq:Ziem},
we obtain the first inequality in \eqref{eq1est}.

\subsection{Proof of the second inequality in \eqref{eq1est}}
Let $\varrho$ be a nonnegative Borel function on the unit hemisphere
$\Sp=\{z\in\mathbb H^n: |z|=1\}$ such that
\begin{equation}\label{admis2}
\int\limits_{\Sp} \varrho(z)^{n-1}\,d\sigma_0(z)\,=\,1\,,
\end{equation}
where $d\sigma_0$ stands for the area element on the unit sphere.
Similarly to the above,
for any $y\in f(\es\setminus B_0),$ we define
\begin{equation*}
\varrho^*(y)\,=\,\frac1{|x-x_0|}\,\varrho\left(\frac{x-x_0}{|x-x_0|}\right)
\left(\frac{\mathcal L_f(x,x_0)}{J_f(x)}\right)^{1/(n-1)}
\,,
\end{equation*}
where $x=f^{-1}(y),$
and set $\varrho^*(y)=\infty$ for $y\in f(B_0)$ and $\varrho^*(y)=0$ at other points.
We now recall that $\Sigma_\es$ is the family of all surfaces $S$ which separate
the boundaries $|x-x_0|=r_0$ and $|x-x_0|=r_1$ in $\es.$
Then we claim that $\varrho^*$ is admissible for $f(\Sigma_\es)=\Sigma_{f(\es)}.$
This is due to \cite[Theorem~1]{Gol10} but a proof is given to see how the definition
of the normal dilatation $T_f(x,x_0)$ works.
The proof of the claim is a little sketchy.
More rigorous arguments may be found in \cite{KR08}.
Note also that by the assumptions of Lemma~\ref{lem:modest} and by \cite[Thm~9.1]{MRSY} the $(n-1)$-dimensional Hausdorff area of the exceptional set $B_0$ for $S$ vanishes
for almost every $S^*\in \Sigma_{f(\es)}.$
To show the claim, it is enough to prove that the inequality
$\int_{S^*}(\varrho^*)^{n-1}d\area^{n-1}\ge 1$ for almost every
$S^*\in \Sigma_{f(\es)}.$
For almost every $f(S)=S^*$ in $f(\Sigma_\es)=\Sigma_{f(\es)}$
we may assume they are regular enough so that the following operations are allowed.
Conventionally, we put $y=f(x)$ for a regular point $x$ of $f$
and let $n$ and $n^*$ be unit normal vectors of $S$ at $x$ and of $S^*$ at $y,$
respectively.
If we denote by $d\sigma$ and $d\sigma^*$ the $(n-1)$-dimensional
area elements of $S$ and $S^*,$
respectively, the rate of the volume change under $f$ at the point $x$ may be
expressed by
$$
J_f(x)=\frac{d\sigma^*(y)}{d\sigma(x)}\times |f'(x)n\cdot n^*|
=\frac{d\sigma^*(y)}{d\sigma(x)}\times |\partial_nf(x)\cdot n^*|.
$$
We next consider the projection $\pi:\mathbb H^n\to\Sp$ defined by
$u=\pi(x)=(x-x_0)/|x-x_0|.$
Then we have
$$
\frac{d\sigma_0(u)}{d\sigma(x)}=\frac{|n\cdot u|}{|x-x_0|^{n-1}}.
$$
Since $\pi(S)=\Sp$ and
$$
\mathcal L_f(x,x_0)\ge |\partial_nf(x)||n\cdot u|
\ge |\partial_nf(x)\cdot n^*||n\cdot u|,
$$
we now compute
\begin{align*}
\int_{S^*}(\varrho^*)^{n-1}d\sigma^*
&=\int_S\frac{\varrho(u)^{n-1}}{|x-x_0|^{n-1}}
\frac{\mathcal L_f(x,x_0)}{J_f(x)}\frac{d\sigma^*(y)}{d\sigma(x)}
d\sigma(x) \\
&=\int_S\varrho(u)^{n-1}\frac{\mathcal L_f(x,x_0)}{|\partial_nf(x)\cdot n^*|}
\frac{d\sigma(x)}{|x-x_0|^{n-1}} \\
&\ge\int_S\varrho(u)^{n-1}\frac{|n\cdot u|d\sigma(x)}{|x-x_0|^{n-1}} \\
&\ge\int_\Sp\varrho(u)^{n-1}d\sigma_0(u)=1.
\end{align*}
Thus the claim has been shown.

The chain rule and the Lusin $(N)$-property of $f$ now yield the inequality
\begin{align*}
\M(\Sigma_{f(\es)})\,&\le\,\int\limits_{f(\es)}\varrho^*(y)^n dm_n(y) \\
&=\,\int\limits_{f(\es)} \frac{1}{|x-x_0|^{n}}\varrho\left(\frac{x-x_0}{|x-x_0|}\right)^n
\left(\frac{\mathcal L_f(x,x_0)^n}{J_f(x)^n}\right)^{1/(n-1)}\,dm_n(y) \\
&=\,\int\limits_\es \varrho\left(\frac{x-x_0}{|x-x_0|}\right)^n
\frac{T_f(x,x_0)}{|x-x_0|^{n}}\,dm_n(x)\,.
\end{align*}
Letting $\varrho(u)=(2/\omega_{n-1})^{1/(n-1)},$ which satisfies (\ref{admis2}),
we obtain the inequality
\begin{equation}\label{Sigma}
\M(\Sigma_{f(\es)})\,\le\,\left(\frac2{\omega_{n-1}}\right)^{n/(n-1)}
\int\limits_\es\frac{T_f(x,x_0)}{|x-x_0|^n}\,dm_n(x)\,.
\end{equation}

The second definition of the modulus of a semiring in \eqref{eq:Ziem}
together with the inequality \eqref{Sigma} yields the second inequality in
(\ref{eq1est}).

\subsection{Proof of \eqref{eq2est}}
To prove \eqref{eq2est} we have to assume that $\mo \es \ge \mo f(\es),$ which has applications in the further discussions. Hence, by \eqref{eq1est},
\begin{equation*}
\frac{\mo \es}{\mo f(\es)}\,\le\,\left(\frac{\mo \es}{\mo f(\es)}\right)^{n-1}\,\le\,\frac{2}{\omega_{n-1}\log(r_1/r_0)}\int\limits_\es
\frac{D_f(x,x_0)}{|x-x_0|^n}\,dm_n(x)\,.
\end{equation*}
This yields
\begin{equation*}
\mo \es-\mo f(\es)\,\le\,\frac{2}{\omega_{n-1}}\frac{\mo f(\es)}{\mo \es}\int\limits_{\es}\frac{D_f(x,x_0)-1}{|x-x_0|^n}\,dm_n(x)\,,
\end{equation*}
and the second inequality in \eqref{eq2est} holds.
The first one in \eqref{eq2est} follows from \eqref{eq1est} immediately.

\begin{rem}
The lower bound in \eqref{eq2est} always holds no matter whether $\mo \es \ge \mo f(\es)$ holds or not. Note also that in the case $\mo \es \ge \mo f(\es),$ the first inequality in \eqref{eq2est} is nontrivial, since $T_f(x,x_0)-1$ can be negative.
\end{rem}

\begin{rem}
Although in Lemma~\ref{lem:modest} we consider the upper half-space $\mathbb H^n,$ obviously the modulus estimates may be easily extended to an arbitrary domain $G$ and a properly embedded semiring $\es\subset G$ under a suitable regularity assumption
of the boundary of $G.$
\end{rem}


\subsection{Lower integral bound}
Here we estimate the modulus of $f(\es)$ for a semiring $\es$ in terms of integrals depending on the angular dilatation $D_f(x,x_0).$

\medskip
\begin{lem}\label{lem:modbound}
Let $f:\mathbb H^n\to\mathbb R^n$ be an orientation-preserving homeomorphism satisfying the assumptions of Lemma~\ref{lem:modest}. Then
for a semiring $\es=\es(x_0; r,R)$ centered at $x_0\in \mathbb H^n$
\begin{equation}\label{modintbound}
\mo f(\es)\,\ge\,\int\limits^R_r\frac{dt}{t\Psi_D(t,x_0)^{1/(n-1)}}\,,
\end{equation}
where
\begin{equation}\label{eq:Psi}
\Psi_D(t,x_0)=\frac{2}{\omega_{n-1}}\int\limits_{\Sp} D_f(x_0+tz, x_0)\,d\sigma_0(z).
\end{equation}
\end{lem}

\begin{proof}
Arguing as in the beginning of the proof of Lemma~\ref{lem:modest}, we obtain
\begin{equation}\label{uppbdir}
\M(f(\Gamma_\es))\,\le\,\int\limits_\es \varrho(|x-x_0|)^n D_f(x,x_0) dm_n(x)\,.
\end{equation}

Since the metric
\begin{equation*}
\varrho(t)\,=\,
\begin{cases}
\displaystyle
\left(t\Psi_D(t,x_0)^{1/(n-1)}\int\limits_r^R\frac{dt}{t\Psi_D(t,x_0)^{1/(n-1)}}\right)^{-1}, &~\text{for } t\in [r,R], \\
0, &~\text{otherwise},
\end{cases}
\end{equation*}
satisfies (\ref{admis1}), the inequality (\ref{uppbdir}) yields
\begin{equation*}
\M(f(\Gamma_S))\,\le\,\frac{\omega_{n-1}}{2}\left(\int\limits_r^R\frac{dt}{t\Psi_D(t,x_0)^{1/(n-1)}}\right)^{1-n}\,.
\end{equation*}
By the first definition in \eqref{eq:Ziem},
we obtain the desired estimate (\ref{modintbound}).
\end{proof}


\subsection{Dominating factor}
The notion of a dominating factor introduced in \cite[p.~882]{GMSV05} for the planar case, will next be extended to higher dimensions $\mathbb R^n,$ $n\ge 2.$

A real valued function $H\,:\,[0,+\infty)\to\mathbb R^n$ is called a \textit{dominating factor} if both of the following conditions hold:
\begin{enumerate}
  \item $H(t)$ is continuous and strictly increasing in $[t_0, +\infty)$ and $H(t)=H(t_0)$ for all $t\in[0,t_0]$ for some $t_0\ge 0;$
  \item the function $e^{H(t)}$ is convex in $t\in [0,+\infty).$
\end{enumerate}

Note that the convexity of $e^H$ implies that $H(t)\to +\infty$ as $t\to\ +\infty.$ Denote by $H^{-1}$ the inverse of $H.$

A dominating factor $H$ is said to be of \textit{divergence type} if
\begin{equation}\label{divtype}
\int\limits_1^{+\infty}\frac{H(t)\,dt}{t^{n/(n-1)}}\,=\,+\infty\,.
\end{equation}
Otherwise, $H$ is of \textit{convergence type}.

An equivalent condition to (\ref{divtype}) can be written as
\begin{equation}\label{invdivtype}
\int\limits_{\tau_1}^{+\infty}\frac{d\tau}{\left[H^{-1}(\tau)\right]^{1/(n-1)}}\,=\,+\infty\,,
\end{equation}
for sufficiently large number $\tau_1.$

Indeed, by the change of variables $\tau=H(t)$ and integration by parts, we have
\begin{equation*}
\int\limits_{\tau_1}^{\tau_2}\frac{d\tau}{\left[H^{-1}(\tau)\right]^{1/(n-1)}}\,=\,\int\limits_{t_1}^{t_2}\frac{dH(t)}{t^{1/(n-1)}}\,=\,
\frac{H(t_2)}{t_2^{1/(n-1)}}-\frac{H(t_1)}{t_1^{1/(n-1)}}+\frac{1}{n-1}\int\limits_{t_1}^{t_2}\frac{H(t)\,dt}{t^{n/(n-1)}}\,,
\end{equation*}
where $\tau_j=H(t_j),$ $j=1,2.$ Therefore, the implication (\ref{divtype}) $\Rightarrow$ (\ref{invdivtype}) is verified immediately. In order to prove the reverse implication assume, on the contrary, that the integral in (\ref{invdivtype}) is finite, whereas (\ref{divtype}) is fulfilled. This implies that $\lim_{t\to\infty} H(t)/t^{1/(n-1)}=\infty,$ i.e. $H(t)>C t^{1/(n-1)}$ for some $C>0$ and all sufficiently large $t.$ Thus, $H(t)/t^{n/(n-1)}>C/t,$ and we reach a contradiction, since the integral in (\ref{invdivtype}) diverges.

\medskip
The function $H(t)=\gamma t,$ where $\gamma$ is a positive constant, serves as an example of a dominating factor of divergence type. Indeed, $H^{-1}(\tau)=\tau/\gamma,$ and
the indefinite integral can be computed as
\begin{equation*}
\int\limits\frac{d\tau}{\left[H^{-1}(\tau)\right]^{1/(n-1)}}\,=\,
\begin{cases}
\gamma \log\tau, &~n=2, \\
\frac{n-1}{n-2}\gamma^{1/(n-1)}\tau^{(n-2)/(n-1)}, &~n\ge 3.
\end{cases}
\end{equation*}
Therefore,
\begin{equation*}
\int\limits_{\tau_1}^{+\infty}\frac{d\tau}{\left[H^{-1}(\tau)\right]^{1/(n-1)}}\,=\,+\infty\,.
\end{equation*}

The following statement provides a lower bound for the modulus of $f(\es(x_0;r,R))$ in terms of a dominating factor. This lower bound is of independent interest, cf. \cite[Lem~2.22]{GMSV05} for the planar case.

\medskip
\begin{lem}\label{lem:modbounddom}
Let $f:\mathbb H^n\to\mathbb R^n$ be an orientation-preserving homeomorphism satisfying the assumptions of Lemma~\ref{lem:modest} and $x_0\in\partial\mathbb H^n.$
Suppose also that a dominating factor $H$ satisfies
\begin{equation}\label{eq:domcond1}
\int\limits_{\es(x_0;r_0e^{-m},r_0)} e^{H(D_f(x,x_0))}\,dm(x)\,\le\,M\,.
\end{equation}
Then
\begin{equation}\label{eq:modintbound}
\mo f(\es(x_0;r_0e^{-m},r_0))\,\ge\,\int\limits^m_{1/n}\frac{dt}{\left[H^{-1}\left(nt+\log\frac{2nM}{\omega_{n-1}r_0^n}\right)\right]^{1/(n-1)}}\,.
\end{equation}
\end{lem}

\begin{proof}
Denoting
\begin{equation*}
h(r)\,=\,\frac{2r^n}{\omega_{n-1}}\int\limits_{\Sp} e^{H(D_f(x_0+rz,x_0))}\, d\sigma_0(z)\,,
\end{equation*}
one can rewrite (\ref{eq:domcond1}) in a form
\begin{equation*}
\int\limits_0^m h(r_0 e^{-t})\,dt\,\le\,\frac{2M}{\omega_{n-1}}\,.
\end{equation*}

Similarly to the proof of \cite[Lem~2.22]{GMSV05}, let $T=\{t\in(0,m):h(r_0 e^{-t})>L\}$ for some $L>0.$ Then the length of $T$ cannot exceed $2M/(\omega_{n-1}L).$

Since $e^H$ is a convex function, Jensen's inequality implies
\begin{equation*}
e^{H(\Psi_D(r,x_0))}\,\le\,\frac{2}{\omega_{n-1}}\int\limits_{\Sp} e^{H(D_f(x_0+rz,x_0))}\,d\sigma_0(z)\,=\,\frac{h(r)}{r^n}\,.
\end{equation*}
where $\Psi_D$ is defined in (\ref{eq:Psi}). This yields,
\begin{equation*}
\Psi_D(r_0 e^{-t},x_0)\,\le\,H^{-1}\left(nt+\log\frac{L}{r_0^n}\right)\qquad\text{for}\quad t\in(0,m)\setminus T\,.
\end{equation*}

Now by Lemma~\ref{lem:modbound} and the last upper bound,
\begin{equation*}
\begin{split}
\mo f(\es)\,&\ge\,\int\limits^{r_0}_{r_0 e^-m}\frac{dr}{r\Psi_D(r,x_0)^{1/(n-1)}}\,=\,
\int\limits^m_0\frac{dt}{\Psi_D(r_0e^{-t},x_0)^{1/(n-1)}}\\&\ge\,\int\limits_{(0,m)\setminus T}\frac{dt}{\left[H^{-1}\left(nt+\log\frac{L}{r_0^n}\right)\right]^{1/(n-1)}}\,
\ge\,\int\limits_{\frac{2M}{\omega_{n-1}L}}^m\frac{dt}{\left[H^{-1}\left(nt+\log\frac{L}{r_0^n}\right)\right]^{1/(n-1)}}\,
\end{split}
\end{equation*}
and setting finally $L=2nM/\omega_{n-1},$ we obtain the desired bound (\ref{eq:modintbound}).
\end{proof}


\section{Boundary correspondence of mappings with finite directional dilatations}\label{Boundext}

In this section we prove results about extending a mapping $f:\mathbb H^n\to \mathbb H^n$ of finite directional dilatations continuously to the boundary.
Moreover, we study the modulus of continuity of the extended mapping in the cases when the extended mapping is Lipschitz or weakly H\"older continuous. Almost all our results rely on applying the Dini condition $\int_0^1 \omega_f(t)/t\,dt<\infty.$ This approach has been utilized for the case when $\omega_f(t)\ge 0$ measures the difference between inner/outer dilatation and 1; e.g. \cite{BGMV03, GMRV98, Resh94}. We apply the Dini condition for essentially weaker cases when $\omega_f(t)=D_f(x,x_0)-1,$ therefore, $\omega_f(t)$ may be negative.


\subsection{Homeomorphic extension to the boundary}
Here we present the main result in our manuscript.
Recall that $\es(x_0; r,R)=\{x\in\mathbb H^n: r\le |x-x_0|\le R\}$
for $x_0\in\partial\mathbb H^n, 0<r<R<+\infty.$

\medskip

\begin{thm}\label{thm:main}
Let $f:\mathbb H^n\to \mathbb H^n$ be a homeomorphism satisfying the assumptions of Lemma~\ref{lem:modest} and $D$ be a domain in the hyperplane $x_n=0.$ Suppose that the both integrals
\begin{equation}\label{eq:main}
\int\limits_{\es(t;r,R)}\frac{T_f(x,t)-1}{|x-t|^n}\,dm(x)\,,\quad
\int\limits_{\es(t;r,R)}\frac{D_f(x,t)-1}{|x-t|^n}\,dm(x)
\end{equation}
are finite for each $t\in D.$ Then $f$ extends to a homeomorphism of $\mathbb H^n\cup D.$
\end{thm}

\begin{proof}
Assume first that $\mo \es \ge \mo f(\es),$ then by (\ref{eq2est}), $\mo f(\es)\to +\infty$ as $r\to 0^+,$ since $\mo \es\to +\infty$ as $r\to 0^+.$

In the case $\mo \es < \mo f(\es),$ the same conclusion $\mo f(\es)\to +\infty$ as $r\to 0^+$ is trivial. Thus, applying a M\"obius transformation from $\mathbb H^n$ onto $\mathbb B^n$ and its inverse which preserve the modulus and then the assertion of Proposition~\ref{prop:lim}, one concludes that $f$ can be extended to a homeomorphism of $\mathbb H^n\cup D$ into $\overline{\mathbb H}^n.$
\end{proof}


\subsection{Distortion of semirings}
The following result is a counterpart of Theorem~\ref{thm:sep} for the case of $\mathbb H^n.$ In addition, the estimates of this type will be applied to studying the regularity features of mappings on the boundary. For the planar case, see \cite[Theorem~2.7]{GSS13}.

\begin{thm}\label{thm:estim}
Let $\es$ be a properly embedded semiring in $\mathbb H^n$
and $V_0$ and $V_1$ be the two connected components of $\mathbb H^n\setminus\es$ bounded and unbounded, respectively.
If, in addition, $\mo\es > A_n,$ then for any point $x_0\in\partial\mathbb H^n\cap V_0,$
\begin{equation}\label{eq:estim}
\sup\limits_{y\in V_0}|y-x_0|\,\le\, C {\rm dist}\,(x_0,V_1)e^{-\mo\es}\,,
\end{equation}
where $C=\exp{A_n}.$
\end{thm}

\begin{proof}
Arguing in the same way as in Lemma~\ref{lem:sep} and applying Theorem~\ref{thm:main} to the symmetrically extended ring $\hat{\es},$ there exists a annular ring $\A=\A(x_0;r,R),$ which is a subring of $\hat{\es} $ such that $\mo \A\ge \mo \es -A_n.$ Now, since $\sup_{y\in V_0}|y-x_0|\le r,$ ${\rm dist}\,(x_0,V_1)\ge R,$ and $\mo \A =R/r,$ we obtain the desired estimate (\ref{eq:estim}).
\end{proof}


\subsection{Lipschitz continuity at the boundary}
The uniform boundedness of the second integral in (\ref{eq:main}) provides a local Lipschitz continuity at the boundary. Recall that a mapping $f$ is called locally Lipschitz continuous on a domain $G,$ if for every compact subset $E$ of $G$ there exists a constant $C=C(E)$ such that $|f(x)-f(x_0)|\le C |x-x_0|$ for any $x, x_0\in E.$ The next result is similar to the earlier theorem for quasiconformal automorphisms of $\mathbb B^n$ normalized by $f(0)=0$ in \cite{GG09}.

\begin{thm}\label{thm:mainLip}
Let $f:\mathbb H^n\to \mathbb H^n$ be a homeomorphism satisfying the assumptions of Theorem~\ref{thm:main} and $D$ be a domain in the hyperplane $x_n=0.$ Suppose that there exist $R>0$ and $M>0$ such that
\begin{equation*}
\int\limits_{S(t;r,R)}\frac{D_f(x,t)-1}{|x-t|^n}\,dm_n(x)\,\le\, M
\end{equation*}
for every $t\in D$ and $r,$ $0<r<R,$ then $f$ extends to a homeomorphism of $\mathbb H^n\cup D$ into $\overline{\mathbb H}^n.$ If $f(D)\subset \partial\mathbb H^n,$ the boundary mapping $f:D\to\partial\mathbb H^n$ is locally Lipschitz continuous.
\end{thm}

\begin{proof}
The scheme of the proof follows the lines of the proof of Theorem~1.2 in \cite{GSS13} given for the planar case.
Note that the existence of boundary extension of $f$ to $D$ follows from Theorem~\ref{thm:main}.
Pick any point $x_0=(x_{01},\ldots,x_{0n})$ satisfying $x_{0n}>R$ and write $y_0=f(x_0).$

First assume that $\mo \es\ge \mo f(\es).$ Then by (\ref{eq2est})
\begin{equation}\label{eq:mainLip1}
\mo f(\es)\,\ge\, \mo \es -\frac{2M}{\omega_{n-1}}\,.
\end{equation}
An appropriate choice of $r_0,$ $\log(R/r_0)-2M/\omega_{n-1}>A_n$ or, equivalently,
\begin{equation*}
r_0\,<\,Re^{-A_n-2M/\omega_{n-1}}\,,
\end{equation*}
allows to conclude that $\mo f(\es) > A_n,$ and, therefore, to apply Theorem~\ref{thm:estim}. Note that $f(\es)$ separates $y_0\in \mathbb H^n$ from $f(t)\in\partial\mathbb H^n,$ thus
${\rm dist}\,(f(t), V_1)\le |f(t)-y_0|.$ Here $V_1$ is the unbounded component of the complement $f(\es)$ in $\mathbb H^n.$ Choose an arbitrary point $x\in \mathbb H^n$ with $|x-t|<r_0$ and set $r=|x-t|.$ Combining (\ref{eq:estim}) with (\ref{eq:mainLip1}) we obtain
\begin{equation*}
|f(x)-f(t)|\,\le\, C {\rm dist}\,(f(t), V_1) e^{-\mo f(\es)}\,\le\, C_1 |f(t)-y_0| |x-t|\,,
\end{equation*}
where $C_1=e^{A_n+2M/\omega_{n-1}}/R.$

In the case $\mo \es < \mo f(\es),$ (\ref{eq:estim}) directly yields
\begin{equation*}
|f(x)-f(t)|\,\le\, C {\rm dist}\,(f(t), V_1) e^{-\mo f(\es)}\,\le\, C_2 |f(t)-y_0| |x-t|\,.
\end{equation*}
Here $C_2=e^{A_n}/R.$ Thus, the mapping $f$ is locally Lipschitz continuous on $D.$
\end{proof}


\subsection{Weak H\"older continuity at the boundary}
The finiteness of the integral average of $(D_f(x,t)-1)/|x-t|^n$ over a half ball centered at $t\in\partial\mathbb H^n$ guarantees a weak H\"older continuity of a self-mapping of $\mathbb H^n$ with finite directional dilatations  up to the boundary. By a weak H\"older continuity with exponent $\alpha$ of a mapping $f$ in a domain $G,$ we mean that there exists a constant $C$ for every $\gamma,$ $0<\gamma<\alpha,$ such that the inequality $|f(x)-f(x_0)|\le C |x-x_0|^\gamma$ holds for any $x, x_0\in G.$ For the same property of quasiconformal automorphisms of $\mathbb B^n$ we refer again to \cite{GG09}.

\begin{thm}\label{thm:mainHol}
Let $f:\mathbb H^n\to \mathbb H^n$ be a homeomorphism satisfying the assumptions of Theorem~\ref{thm:main} and $D$ be a domain in the hyperplane $x_n=0.$ Suppose that for  $0<\alpha\le 1$ we have
\begin{equation}\label{eq:mainHol}
\limsup_{R\to 0^+}\frac{2}{\Omega_n R^n}\int\limits_{\es(t;0,R)}\left(D_f(x,t)-1\right)\,dm_n(x)\,\le\, \frac{1}{\alpha^{n-1}}-1
\end{equation}
uniformly for every $t\in D.$ Then $f$ extends to a homeomorphism of $\mathbb H^n\cup D$ into $\overline{\mathbb H}^n,$ and if $f(D)\subset \partial\mathbb H^n,$ the boundary mapping $f:D\to\partial\mathbb H^n$ is locally weakly H\"older continuous on $D$ with exponent $\alpha.$
\end{thm}

\begin{proof}
Denote
\begin{equation*}
\omega(t;R)\,=\,\frac{2}{\Omega_nR^n}\int\limits_{\es(t;0,R)}\left(D_f(x,t)-1\right)\,dm_n(x)\,=\,\frac{2}{\Omega_nR^n}\int\limits_{\es(t;0,R)}D_f(x,t)\,dm_n(x)-1\,,
\end{equation*}
and
\begin{equation*}
P_f(t;r,R)\,=\,\frac{2}{\omega_{n-1}\log(R/r)}\int\limits_{\es(t;r,R)}\frac{D_f(x,t)}{|x-t|^n}\,dm_n(x)\,.
\end{equation*}
Then arguing similarly to (2.19) in \cite{GG09}, we obtain
\begin{equation}\label{eq:mainHol1}
\left(P_f(t;r,R)-1\right)\log(R/r)\,=\,\frac{\omega(t;R)-\omega(t;r)}{n}+\int\limits_r^R\frac{\omega(t;s)}{s}\,ds\,.
\end{equation}

Note now that the condition (\ref{eq:mainHol}) is equivalent to
\begin{equation*}
\lim\limits_{R\to 0^+} \omega(t;R)\,\le\, \frac{1}{\alpha^{n-1}}-1
\end{equation*}
uniformly for $t\in D.$ Thus, for a compact set $D_0\subset D$ and arbitrary $\gamma,$ $0<\gamma<\alpha,$ there exists $R>0$ such that $\omega(t;s)\le 1/\gamma^{n-1}-1$ for $t\in D_0$ and $0<s\le R;$ cf. \cite[p. 960]{GSS13}. The function $\omega(t;s)$ is bounded from below and from above for $0<s\le R,$ and therefore, we rewrite (\ref{eq:mainHol1}) as
\begin{equation*}
\left(P_f(t;r,R)-1\right)\log(R/r)\,\le\, O(1)+\left(1/\gamma^{n-1}-1\right)\log(R/r)\,,\quad\text{as}\quad r\to 0\,,
\end{equation*}
which implies
\begin{equation*}
P_f(t;r,R)\,\le\, 1/\gamma^{n-1}+o(1),\quad\text{as}\quad r\to 0\,.
\end{equation*}
Then by the first inequality in (\ref{eq1est}),
\begin{equation*}
\frac{\mo f(\es)}{\log(R/r)}\,\ge\, P_f(t;r,R)^{1/(1-n)}\,,
\end{equation*}
and choosing sufficiently small $r_0,$ one can apply Theorem~\ref{thm:estim} and reach the estimate
\begin{equation*}
|f(y)-f(t)|\,\le\, C |y-t|^\gamma
\end{equation*}
for $y\in D,$ $t\in D_0,$ provided that $|y-t|\le r_0.$
\end{proof}


\subsection{Homeomorphic extension to $x_n=0$}
We next prove a counterpart of Theorem~\ref{thm:exten} for the case of the upper half space. The proof follows the ideas  of
\cite[Thm~3.4]{GSS13}.

\begin{lem}\label{lem:exten}
An automorphism $f$ of the upper half space $\mathbb H^n$ admits a homeomorphic extension to $\overline{\mathbb H}^n$ if and only if for each $t\in\partial\mathbb H^n,$
\begin{equation*}
\lim_{r\to 0^+}\mo f(\es(t; r,R))\,=\,+\infty
\end{equation*}
for some $R=R(t)>0.$
\end{lem}


\subsection{Modulus of continuity} In this subsection we establish estimates for the modulus of continuity involving a dominating factor of divergence type.
First we present a sufficient condition for the continuous extension to the boundary for self-homeomorphisms of $\mathbb H^n.$

\begin{thm}\label{thm:domfac}
Let $f:\mathbb H^n\to \mathbb H^n$ be a homeomorphism satisfying the assumptions of Lemma~\ref{lem:modest} and $x_0\in\partial \mathbb H^n.$ Suppose that for some positive constants $\gamma$ and $M=M(x_0),$
\begin{equation*}
\int\limits_{\es(x_0;r_0e^{-m},r_0)} e^{\gamma D_f(x,x_0)}\,dm(x)\,\le\,M\,.
\end{equation*}
Then $f$ continuously extends to $x_0.$
\end{thm}

\begin{proof}
The assumption of Theorem~\ref{thm:domfac} yields that $f$ satisfies the conditions of Lem\-ma~\ref{lem:modbounddom} with a dominating factor of divergence type $H(t)=\gamma t.$ Then, by (\ref{eq:modintbound})
\begin{equation*}
\mo f(\es(x_0;r_0e^{-m},r_0))\,\ge\,\int\limits^m_{1/n}\frac{dt}{\left[H^{-1}\left(nt+\sigma\right)\right]^{1/(n-1)}}\,,
\end{equation*}
where $\sigma=\log\frac{2nM}{\omega_{n-1}r_0^n}.$ A straightforward calculation gives
\begin{equation}\label{eq:modintbound1}
\begin{split}
\mo f(\es(x_0;r_0e^{-m},r_0))&\ge\gamma^{1/(n-1)}\int\limits^m_{1/n}\left(nt+\sigma\right)^{1/(1-n)}\,dt
\\ &=
\begin{cases}
C_1\left[(nm+\sigma)^\mu-(1+\sigma)^\mu\right], & n\ge 3, \\
C_2\log\frac{nm+\sigma}{1+\sigma}, & n=2,
\end{cases}
\end{split}
\end{equation}
where $C_1=\frac{(n-1)\gamma^{1/(n-1)}}{n(n-2)},$ $C_2=\frac{\gamma^{1/(n-1)}}{n}$ and $\mu=\frac{n-2}{n-1}.$ Letting $m\to\infty,$ we have
\begin{equation*}
\lim\limits_{m\to\infty} \mo f(\es(x_0;r_0e^{-m},r_0))\,=\,+\infty
\end{equation*}
in both cases, therefore, the desired assertion follows from Lemma~\ref{lem:exten}.
\end{proof}

\medskip
Let $x_0\in\partial\mathbb H^n$ and $x_1\in\mathbb H^n$ be two arbitrary points. Denote $|x_1-x_0|=r_0e^{-m},$ where $m>0$ and $r_0>0$ can be precisely defined later. Consider the  semiring $\es=\es(x_0;r_0e^{-m},r_0)$ and an orientation-preserving homeomorphism $f:\mathbb H^n\to\mathbb H^n.$ Then $f(\es)$ is a properly embedded semiring in $\mathbb H^n.$ Assume, in addition, that $\mo f(\es)>A_n.$ Denoting by $V_0$ and $V_1$ the connected components of $\mathbb H^n\setminus f(\es)$ ($\infty\in V_1$), Theorem~\ref{thm:estim} implies
\begin{equation}\label{eq:estim1}
\sup\limits_{y\in V_0}|y-f(x_0)|\,\le\, C {\rm dist}\,(f(x_0),V_1)e^{-\mo f(\es)}\,,
\end{equation}
where $C=\exp(A_n).$

\begin{thm}\label{thm:modcont}
Let $f:\mathbb H^n\to \mathbb H^n$ be a homeomorphism satisfying the assumptions of Theorem~\ref{thm:domfac} and $x_0\in \partial \mathbb H^n.$ Then for any $x_1\in \mathbb H^n,$
the following modulus of continuity estimates
\begin{equation}\label{eq:modcont2}
|f(x_1)-f(x_0)|\,\le\,\alpha\left[\log\frac{r_0}{|x_1-x_0|}\right]^{-C_2}\,,\qquad \text{if}\quad n=2\,,
\end{equation}
\begin{equation}\label{eq:modcont3}
\log|f(x_1)-f(x_0)|\,\le\,-\beta\left[\log\frac{r_0}{|x_1-x_0|}\right]^\mu+\delta\,,\qquad \text{if}\quad n\ge 3\,,
\end{equation}
hold, where $\alpha$ and $\delta$ are constants depending on ${\rm dist}\,(f(x_0),V_1)$ and $A_n,$ $\mu=(n-2)/(n-1),$ $\beta=C_1 n^\mu,$  and $C_1$ and $C_2$ are defined in (\ref{eq:modintbound1}).
\end{thm}

\begin{proof}
Let $x_0\in \partial \mathbb H^n$ and $x_1\in \mathbb H^n$ be two arbitrary points. Denote $|x_1-x_0|=r_0e^{-m},$ where the constant $m>0$ can be chosen in the following way.
By Theorem~\ref{thm:domfac},
\begin{equation*}
\mo f(\es)\,\ge\,\gamma^{1/(n-1)}\int\limits^m_{1/n}\left(nt+\sigma\right)^{1/(1-n)}\,dt
\end{equation*}
for the dominating factor of divergence type $H(t)=\gamma t.$ Since the integral in the right-hand side tends to $\infty$ as $m\to\infty,$ one can choose $m>0$ such that $\mo f(\es)>A_n.$ Fix such $m$ and $r_0=|x_1-x_0|e^m.$ The above lower bound by $A_n$ for $\mo f(\es)$ yields (\ref{eq:estim1}).

Clearly, $|f(x_1)-f(x_0)|\le\sup_{y\in V_0}|y-f(x_0)|,$ therefore, using (\ref{eq:modintbound1}) together with (\ref{eq:estim1}) for $n=2$ provides
\begin{equation*}
|f(x_1)-f(x_0)|\,\le\, C {\rm dist}\,(f(x_0),V_1)e^{-C_2\log\frac{nm+\sigma}{1+\sigma}}\,.
\end{equation*}
By a simple chain of upper bounds, we obtain
\begin{equation*}
|f(x_1)-f(x_0)|\,\le\,\,\widetilde{C}(nm+\sigma)^{-C_2}\,\le
\,\alpha m^{-C_2}\,=\,\alpha\left[\log\frac{r_0}{|x_1-x_0|}\right]^{-C_2}\,,
\end{equation*}
where $\alpha={\rm dist}\,(f(x_0),V_1)\exp(A_n-C_2\log{n}) .$

For higher dimensions, i.e. $n\ge 3,$
\begin{equation*}
\begin{split}
|f(x_1)-f(x_0)|\,&\le\,\widetilde{C} e^{-C_1\left[(nm+\sigma)^\mu-(1+\sigma)^\mu\right]}\,\le\,\widehat{C} e^{-C_1(nm+\sigma)^\mu}\\
&\le\,\widehat{C} e^{-C_1(nm)^\mu}\,=\,\widehat{C}e^{-\beta m^\mu}\,,
\end{split}
\end{equation*}
where $\mu=(n-2)/(n-1)$ and $\beta=C_1 n^\mu.$ Finally, by taking $\log$,
\begin{equation*}
\log|f(x_1)-f(x_0)|\,\le\,-\beta\left[\log\frac{r_0}{|x_1-x_0|}\right]^\mu+\delta\,,
\end{equation*}
where $\delta=A_n+C_1(1+\sigma)^\mu+\log {\rm dist}\,(f(x_0),V_1).$ This completes the proof.
\end{proof}


\subsection{Behavior at infinity}
The following result shows that under an appropriate condition on asymptotic behavior of the integral of a Teichm\"uller-Wittich-Belinski\u\i\ type the mappings admit continuous extensions to infinity; cf. \cite[Lemma~1.4]{GSS13}.

\begin{lem}\label{lem:infty}
Let $f:\mathbb H^n\to \mathbb H^n$ be a homeomorphism satisfying the assumptions of Theorem~\ref{thm:main} and
\begin{equation}\label{eq:infty}
\lim\limits_{R\to\infty}\frac{1}{\left(\log R\right)^2}\int\limits_{\es(0;r_0,R)}\frac{D_f(x,0)-1}{|x|^n}\, dm_n(x)\,=\,0
\end{equation}
for some $r_0>0.$ Then $f$ extends continuously to infinity.
\end{lem}

\begin{proof}
In the proof of Theorem~\ref{thm:mainHol} we showed that
\begin{equation*}
P_f(0;r_0,R)-1\,=\,\frac{2}{\omega_{n-1}\log(R/r_0)}\int\limits_{\es(t;r,R)}\frac{D_f(x,0)-1}{|x|^n}\,dm_n(x).
\end{equation*}
Thus, due to (\ref{eq:infty}), $P_f(0;r_0,R)=o(\log R)$ as $R\to +\infty,$ and the desired assertion follows from (\ref{eq1est}) as $R\to +\infty$ and then by Lemma~\ref{lem:exten}.
\end{proof}







\bigskip
\noindent
{\bf Acknowledgements.}
The memory of Professor Lawrence (Larry) Zalcman, a distinguished mathematician, an outstanding editor-in-chief for 30 years of the brilliant analysis journal ``Journal d'Analyse Math\'ematique'', an attentive colleague, an extremely gorgeous person and friend, will live in the hearts of his many friends.


\bigskip
\noindent
{\bf Declarations.}

Conflict of Interest: None.

Ethical Approval: Not applicable.

Data Availability: No datasets were generated or analyzed during this study.



\begin{thebibliography} {99}

\bibitem {AVV97}
{\sc G.~D.~Anderson, M.~K.~Vamanamurthy and M.~K.~Vuorinen}.
{\em Conformal invariants, inequalities, and quasiconformal maps}, John Wiley \& Sons, Inc., New York, 1997.

\bibitem {AW09}
{\sc F.~G.~Avkhadiev and K.-J.~Wirths}, {\em Schwarz-Pick type inequalities}. Frontiers in Mathematics. Birkh\"auser Verlag, Basel, 2009.

\bibitem{BGMV03}
{\sc Ch.~J.~Bishop, V.~Ya.~Gutlyanski\u\i, O.~Martio and M.~Vuorinen},
{\em On conformal dilatation in space},
Int. J. Math. Math. Sci. (2003), no. 22, 1397--1420.

\bibitem{Fug57}
{\sc B.~Fuglede}, {\em  Extremal length and functional completion},
Acta Math. \textbf{98} (1957), 171--219.

\bibitem{GMP17}
{\sc F.~W.~Gehring, G.~J.~Martin and B.~P.~Palka}.
{\em An introduction to the theory of higher-dimensional quasiconformal mappings}.
Mathematical Surveys and Monographs, 216. American Mathematical Society, Providence, RI, 2017.

\bibitem{Gol10}
{\sc A.~Golberg,}  {\em Directional dilatations in space}, Complex Var.
Elliptic Equ., \textbf{55} (2010), 13--29.

\bibitem{Gol11}
{\sc A.~Golberg,}  {\em Homeomorphisms with integrally restricted moduli},
Contemp. Math., \textbf{553} (2011), 83--98.

\bibitem{Gol18}
{\sc A.~Golberg,}  {\em Extremal bounds of Teichm\"uller-Wittich-Belinski\u\i\ type for planar quasiregular mappings},
Fields Inst. Commun. \textbf{81},
Springer, New York, 2018, 173--199.

\bibitem{GSV20}
{\sc A.~Golberg, T.~Sugawa and M.~Vuorinen}, {\em Teichm\"uller's theorem in higher dimensions and its applications},
Comput. Methods Funct. Theory \textbf{20} (2020), no.~3--4, 539--558.

\bibitem{GG09}
{\sc V.~Ya.~Gutlyanski\u\i\ \ and A. Golberg},  {\em On Lipschitz continuity
of quasiconformal mappings in space}, J. Anal. Math. \textbf{109}
(2009), 233--251.

\bibitem{GMRV98}
{\sc V.~Gutlyanski\u\i, O.~Martio, V.~Ryazanov and M.~Vuorinen},
{\em  On local injectivity and asymptotic linearity of quasiregular mappings},
Studia Math. \textbf{128} (1998), no.~3, 243--271.

\bibitem{GMSV05}
{\sc V.~Gutlyanski\u\i, O.~Martio, T.~Sugawa and M.~Vuorinen},  {\em On the
degenerate Beltrami equation}, Trans. Amer. Math. Soc. \textbf{357}
(2005), no. 3, 875--900.

\bibitem{GSS13}
{\sc V.~Gutlyanski\u\i, K.~Sakan and T.~Sugawa,}
{\em On $\mu$-conformal homeomorphisms and boundary correspondence},
Complex Var. Elliptic Equ. \textbf{58} (2013), no. 7, 947--962.

\bibitem{HKV20}
{\sc P.~Hariri, R.~Kl\'{e}n and M.~Vuorinen,}
{\em Conformally invariant metrics and quasiconformal mappings}.
Springer Monogr. Math., Springer, Cham, 2020.

\bibitem{HK14}
{\sc S.~Hencl and P.~Koskela,} {\em Lectures on mappings of finite distortion}.
Lecture Notes in Mathematics, 2096. Springer, Cham, 2014.

\bibitem{KR08}
{\sc D.~Kovtonyk and V.~Ryazanov}, {\em On the theory of mappings with finite area distortion}.
J. Anal. Math. \textbf{104} (2008), 291--306.


\bibitem{MRV69}
{\sc O.~Martio, S.~Rickman and J. V\"ais\"al\"a}, {\em Definitions for quasiregular mappings}.
Ann. Acad. Sci. Fenn. Ser. A I \textbf{448} (1969), 1--40.

\bibitem{MRSY}
{\sc O.~Martio, V.~Ryazanov, U.~Srebro, and E.~Yakubov}, {\em Moduli in modern mapping theory}.
Springer Monographs in Mathematics. Springer, New York, 2009.

\bibitem {Pon95}
{\sc S.~P.~Ponomarev}, {\em The $N^{-1}$-property of mappings, and Luzin's $(N)$ condition}.
Math. Notes \textbf{58} (1995), no. 3--4, 960--965 (1996).

\bibitem{Resh89}
{\sc Yu.~G.~Reshetnyak}, {\em Space mappings with bounded distortion}.
Translations of Mathematical Monographs, 73. American Mathematical Society, Providence, RI, 1989.

\bibitem{Resh94}
{\sc Yu.~G.~Reshetnyak}, {\em Stability theorems in geometry and analysis}.
Math. Appl., 304, Kluwer Academic Publishers Group, Dordrecht, 1994.

\bibitem{Ric93}
{\sc S.~Rickman}, {\em Quasiregular mappings}.
Results in Mathematics and Related Areas (3), 26. Springer-Verlag, Berlin, 1993.

\bibitem{Sug10}
{\sc T.~Sugawa,}  {\em Modulus techniques in geometric function theory}, J. Anal. \textbf{18} (2010), 373--397.

\bibitem {Teng14}
{\sc V.~Tengvall}, {\em Differentiability in the Sobolev space $W^{1,n-1}$},
Calc. Var. Partial Differential Equations \textbf{51} (2014), no. 1--2, 381--399.

\bibitem{Vai71}
{\sc J.~V\"ais\"al\"a,} {\em Lectures on n-dimensional quasiconformal
mappings}, Lecture Notes in Mathematics, Vol.~229. Springer-Verlag, Berlin-New York, 1971.

\bibitem{ZYRH24}
{\sc Q.~Zhou, Z.~Yang, A.~Rasila and Y.~He},
{\em Modulus characterizations of bilipschitz mappings},
J. Geom. Anal. \textbf{34} (2024), no.~6, Paper No. 155, 22 pp.

\bibitem {Ziem67}
{\sc W.~P.~Ziemer,}  {\em Extremal length and conformal capacity},
Trans. Amer. Math. Soc. \textbf{126} (1967), 460--473.


\end{thebibliography}
\end{document}